\documentclass[a4paper]{article}%
\usepackage{amssymb}
\usepackage{amsmath}
\usepackage{amsfonts}
\usepackage{graphicx}
\usepackage{geometry}%
\newtheorem{theorem}{Theorem}
\newtheorem{proposition}[theorem]{Proposition}

\newenvironment{proof}[1][Proof]{\noindent\textbf{#1.} }{\ \rule{0.5em}{0.5em}}
\begin{document}

\title{De- Moivre's and Euler Formulas for Matrices of Split Quaternions}
\author{Melek Erdo\u{g}du \thanks{Corresponding Author}
\and Mustafa \"{O}zdemir }
\maketitle

\begin{abstract}
In this paper, real matrix representations of split quaternions are examined
in terms of the casual character of quaternion. Then, we give De-Moivre' s
formula for real matrices of timelike and spacelike split quaternions,
separately. Finally, we state the Euler theorem for real matrices of pure
split quaternions.



\end{abstract}

\section{Introduction}

\noindent Sir William Rowan Hamilton discovered the quaternions in 1843, which
was one of the his best contribution made to mathematical science. This
discovery is a way of extending complex numbers to higher dimensions. So the
set of quaternions, which was introduced by Hamilton, can be represented as%
\[
\mathbb{H}=\{q=q_{0}+q_{1}i+q_{2}j+q_{3}k;\text{ \ }q_{0},q_{1},q_{2},q_{3}\in%
\mathbb{R}
\}
\]
where%
\[
i^{2}=j^{2}=k^{2}=-1\text{ and }ijk=-1.
\]
The set of quaternions is a member of noncommutative division algebra
\cite{kantor}.

\bigskip

\noindent We know that real quaternion algebra is isomorphic to a real
$4\times4$ matrix algebra. The real matrix representations of quaternions are
investigated in \cite{zhang-97}, \cite{farebrother} and \cite{grob}. Euler and
De-Moivre formulas for real matrices associated with quaternions are studied
in \cite{yayli}.

\bigskip

\noindent In 1849, James Cockle introduced the set of split\ quaternions, also
known as coquaternions. The real algebra of split quaternions, denoted by
$\widehat{\mathbb{H}},$ is a four dimensional vector space over the real
field\ $%
\mathbb{R}
$ of real numbers with a basis $\{1,i,j,k\}$ satisfying%
\[
i^{2}=-1,\text{ }j^{2}=k^{2}=1\text{ and }ijk=1.
\]
The set of split quaternions is noncommutative, too. Unlike quaternion
algebra, the set of split quaternions contains zero divisor, nilpotent
elements and nontrivial idempotents \cite{kula}, \cite{ozdemir-09}.

\bigskip

\noindent Split quaternions is a recently developing topic. There are some
studies related to geometric applications of split quaternions such as
\cite{kula}, \cite{ozdemir-05}, \cite{split} and \cite{ozdemir-06}.
Particularly, the geometric and physical applications of quaternions require
solving quaternionic equations. Therefore, there are many studies on
quaternionic and split quaternionic equations. For example; the rearrangement
method of solving two sided linear split quaternionic equations is given in
\cite{erdogdu 3} and De Moivre's formula is used to find the roots of split
quaternion in work \cite{ozdemir-09}. Furthermore, $2\times2$ complex matrix
representation of any split quaternion is presented in \cite{alagoz} and the
left and right real matrix representations of split quaternions are studied in
\cite{kula}.

\bigskip

\noindent In this paper, we will investigate real matrix representations of
split quaternions. First, we present a brief introduction of split
quaternions. Then, we examine real matrices associated with split quaternions
depending on the casual character of the quaternion. Moreover, we give
De-Moivre' s formula for real matrices of timelike and spacelike split
quaternions, separately. Finally, we state the Euler theorem for real matrices
of pure split quaternions.\newpage

\section{Preliminaries}

\noindent In this section, we present an introduction to split quaternions for
the necessary background \cite{kula}, \cite{ozdemir-09}, \cite{split}.

\subsection{Split quaternions}

\noindent The set split quaternions can be represented as%
\[
\widehat{\mathbb{H}}=\{q=q_{0}+q_{1}i+q_{2}j+q_{3}k;\text{ \ }q_{0}%
,q_{1},q_{2},q_{3}\in%
\mathbb{R}
\}
\]
where the product table is:%
\[%
\begin{tabular}
[c]{|c|c|c|c|c|}\hline
& $1$ & $i$ & $j$ & $k$\\\hline
$1$ & $1$ & $i$ & $j$ & $k$\\\hline
$i$ & $i$ & $-1$ & $k$ & $-j$\\\hline
$j$ & $j$ & $-k$ & $1$ & $-i$\\\hline
$k$ & $k$ & $j$ & $i$ & $1$\\\hline
\end{tabular}
\ \ \ \ \ \
\]
We write any split quaternion in the form $q=q_{0}+q_{1}i+q_{2}j+q_{3}%
k=S_{q}+\overrightarrow{V_{q}}$ where $S_{q}=q_{0}$ denotes the scalar part of
$q$ and $\overrightarrow{V_{q}}=q_{1}i+q_{2}j+q_{3}k$ denotes vector part of
$q.$ If $S_{q}=0$ then $q$ is called pure split quaternion. The set of pure
split quaternions are identified with the Minkowski 3 space. The Minkowski 3
space is Euclidean 3 space with the Lorentzian inner product
\[
\left\langle \overrightarrow{u},\overrightarrow{v}\right\rangle _{\mathbb{L}%
}=-u_{1}v_{1}+u_{2}v_{2}+u_{3}v_{3}%
\]
where $\overrightarrow{u}=(u_{1},u_{2},u_{3})$ and $\overrightarrow{v}%
=(v_{1},v_{2},v_{3})\in\mathbb{E}^{3}$ and denoted by $\mathbb{E}_{1}^{3}.$ We
say that a vector $\overrightarrow{u}$ in $\mathbb{E}_{1}^{3}$ is spacelike,
timelike or null if $\left\langle \overrightarrow{u},\overrightarrow
{u}\right\rangle _{\mathbb{L}}>0,$ $\left\langle \overrightarrow
{u},\overrightarrow{u}\right\rangle _{\mathbb{L}}>0$ or $\left\langle
\overrightarrow{u},\overrightarrow{u}\right\rangle _{\mathbb{L}}=0,$
respectively. The conjugate of a split quaternion $q=q_{0}+q_{1}i+q_{2}%
j+q_{3}k$ is denoted by $\overline{q}$ and it is $\overline{q}=S_{q}%
-\overrightarrow{V_{q}}=q_{0}-q_{1}i-q_{3}j-q_{4}k.$ For any $p,q\in
\widehat{\mathbb{H}}$, the sum and product of split quaternions $p$ and $q$
are%
\begin{align*}
p+q  &  =S_{p}+S_{q}+\overrightarrow{V_{p}}+\overrightarrow{V_{q}},\\
pq  &  =S_{p}S_{q}+\left\langle \overrightarrow{V_{p}},\overrightarrow{V_{q}%
}\right\rangle _{\mathbb{L}}+S_{p}\overrightarrow{V_{q}}+S_{q}\overrightarrow
{V_{p}}+\overrightarrow{V_{p}}\times_{\mathbb{L}}\overrightarrow{V_{q}},
\end{align*}
respectively. Here $\times_{\mathbb{L}}$ denotes Lorentzian vector product and
is defined as%
\[
\overrightarrow{u}\times_{\mathbb{L}}\overrightarrow{v}=\left\vert
\begin{array}
[c]{ccc}%
-e_{1} & e_{2} & e_{3}\\
u_{1} & u_{2} & u_{3}\\
v_{1} & v_{2} & v_{3}%
\end{array}
\right\vert ,
\]
for vectors $\overrightarrow{u}=(u_{1},u_{2},u_{3})$ and $\overrightarrow
{v}=(v_{1},v_{2},v_{3})$ of Minkowski 3 space. And norm of the split
quaternion $q$ is defined by%
\[
N_{q}=\sqrt{\left\vert q\overline{q}\right\vert }=\sqrt{\left\vert q_{0}%
^{2}+q_{1}^{2}-q_{2}^{2}-q_{3}^{2}\right\vert }.
\]
If $N_{q}=1$ then $q$ is called unit split quaternion and $q_{0}=q/N_{q}$ is a
unit split quaternion for $N_{q}\neq0.$ And the product%
\[
I_{q}=q\overline{q}=\overline{q}q=q_{0}^{2}+q_{1}^{2}-q_{2}^{2}-q_{3}%
^{2}\medskip
\]
determines the character of a split quaternion. A split quaternion is
spacelike, timelike or lightlike (null) if $I_{q}<0,$ $I_{q}>0$ or $I_{q}=0,$
respectively. Polar forms of split quaternions are defined as follows:

\noindent\textbf{i.\qquad}Every timelike split quaternion with spacelike
vector part can be written in the form
\[
q=N_{q}(\cosh\theta+\overrightarrow{\varepsilon}\sinh\theta)
\]
where $\overrightarrow{\varepsilon}$ is spacelike unit vector in
$\mathbb{E}_{1}^{3}.$

\noindent\textbf{ii.\qquad}Every timelike split quaternion with timelike
vector part can be written in the form
\[
q=N_{q}(\cos\theta+\overrightarrow{\varepsilon}\sin\theta)
\]
where $\overrightarrow{\varepsilon}$ is timelike unit vector in $\mathbb{E}%
_{1}^{3}.$

\noindent\textbf{iii.\qquad}Every spacelike split quaternion can be written in
the form
\[
q=N_{q}(\sinh\theta+\overrightarrow{\varepsilon}\cosh\theta)
\]
where $\overrightarrow{\varepsilon}$ is spacelike unit vector in
$\mathbb{E}_{1}^{3}.$\newpage

\subsection{De-Moivre's formula for split quaternions}

\noindent Euler and De-Moivre formula for complex numbers are generalized for
quaternions in \cite{cho 1}, \cite{cho 2}. Similarly to the discussions in
\cite{cho 1} and \cite{cho 2}, De-Moivre formula and roots of split
quaternions are investigated in \cite{ozdemir-09} by considering the casual
characters of the quaternion. De-Moivre formulas for split quaternions, which
are stated in \cite{ozdemir-09}, can be given as follows;

\noindent\textbf{i.}\qquad If $q=N_{q}(\cosh\theta+\overrightarrow
{\varepsilon}\sinh\theta)$ be a timelike split quaternion with spacelike
vector part then%
\[
q^{n}=(N_{q})^{n}(\cosh n\theta+\overrightarrow{\varepsilon}\sinh n\theta),
\]

\noindent\textbf{ii.\qquad}If $q=N_{q}(\cos\theta+\overrightarrow{\varepsilon
}\sin\theta)$ be a timelike split quaternion with timelike vector part then%
\[
q^{n}=(N_{q})^{n}(\cos n\theta+\overrightarrow{\varepsilon}\sin n\theta),
\]

\noindent\textbf{iii.}\qquad If $q=N_{q}(\sinh\theta+\overrightarrow
{\varepsilon}\cosh\theta)$ be a spacelike split quaternion then%
\[
q^{n}=\left\{
\begin{tabular}
[c]{ll}%
$(N_{q})^{n}(\sinh\theta+\overrightarrow{\varepsilon}\cosh\theta),$ & $n$ is
odd\bigskip\\
$(N_{q})^{n}(\cosh n\theta+\overrightarrow{\varepsilon}\sinh n\theta)$ & $n$
is even
\end{tabular}
\ \right\}
\]
for $n\in\mathbb{N}$.

\subsection{$4\times4$ Real matrix representations of split quaternions}

\noindent For any $q\in\widehat{\mathbb{H}},$ consider the linear map
\[
g_{q}:\widehat{\mathbb{H}}\rightarrow\widehat{\mathbb{H}}%
\]
defined as
\[
g_{q}(p)=qp.
\]
This map is bijective and for $q=q_{0}+q_{1}i+q_{2}j+q_{3}k\in\widehat
{\mathbb{H}},$ we have%
\begin{align*}
g_{q}(1)  &  =q_{0}+q_{1}i+q_{2}j+q_{3}k,\\
g_{q}(i)  &  =-q_{1}+q_{0}i+q_{3}j-q_{2}k,\\
g_{q}(j)  &  =q_{2}+q_{3}i+q_{0}j+q_{1}k,\\
g_{q}(k)  &  =q_{3}-q_{2}i-q_{1}j+q_{0}k.
\end{align*}
Using this map, we may define an isomorphism between $\widehat{\mathbb{H}}$
and algebra of the matrices%
\[
\left\{  \left[
\begin{array}
[c]{cccc}%
q_{0} & -q_{1} & q_{2} & q_{3}\\
q_{1} & q_{0} & q_{3} & -q_{2}\\
q_{2} & q_{3} & q_{0} & -q_{1}\\
q_{3} & -q_{2} & q_{1} & q_{0}%
\end{array}
\right]  :q_{0},q_{1},q_{2},q_{3}\in\mathbb{%
\mathbb{R}
}\right\}  .
\]
The above obtained $4\times4$ real matrix for the split quaternion $q$ is
denoted by%
\[
\mathbf{L}_{q}=\left[
\begin{array}
[c]{cccc}%
q_{0} & -q_{1} & q_{2} & q_{3}\\
q_{1} & q_{0} & q_{3} & -q_{2}\\
q_{2} & q_{3} & q_{0} & -q_{1}\\
q_{3} & -q_{2} & q_{1} & q_{0}%
\end{array}
\right]
\]
and called the left matrix representation of $q.$

\begin{proposition}
\cite{kula} For any $p,q\in\widehat{\mathbb{H}}$ and $r\in%
\mathbb{R}
,$ the followings are satisfied;
\end{proposition}

\begin{description}
\item[i)] $\mathbf{L}_{p+q}\mathbf{=L}_{p}+\mathbf{L}_{q},$

\item[ii)] $\mathbf{L}_{pq}=\mathbf{L}_{p}\mathbf{L}_{q},$

\item[iii)] $\mathbf{L}_{rp}=r\mathbf{L}_{p},$

\item[iv)] $\mathbf{L}_{1}=\mathbf{I}_{4}.$\newpage
\end{description}

\noindent Similar to the previous discussion, consider the linear map
\[
f_{q}:\widehat{\mathbb{H}}\rightarrow\widehat{\mathbb{H}}%
\]
defined as
\[
f_{q}(p)=pq.
\]
This map is also bijective and for $q=q_{0}+q_{1}i+q_{2}j+q_{3}k\in
\widehat{\mathbb{H}},$ we have%
\begin{align*}
f_{q}(1)  &  =q_{0}+q_{1}i+q_{2}j+q_{3}k,\\
f_{q}(i)  &  =-q_{1}+q_{0}i-q_{3}j+q_{2}k,\\
f_{q}(j)  &  =q_{2}-q_{3}i+q_{0}j-q_{1}k,\\
f_{q}(k)  &  =q_{3}+q_{2}i+q_{1}j+q_{0}k.
\end{align*}
By this map, we may define an isomorphism between $\widehat{\mathbb{H}}$ and
algebra of the matrices%
\[
\left\{  \left[
\begin{array}
[c]{cccc}%
q_{0} & -q_{1} & q_{2} & q_{3}\\
q_{1} & q_{0} & -q_{3} & q_{2}\\
q_{2} & -q_{3} & q_{0} & q_{1}\\
q_{3} & q_{2} & -q_{1} & q_{0}%
\end{array}
\right]  :q_{0},q_{1},q_{2},q_{3}\in\mathbb{%
\mathbb{R}
}\right\}  .
\]
We denote above corresponding $4\times4$ real matrix for any split quaternion
$q$ by%
\[
\mathbf{R}_{q}=\left[
\begin{array}
[c]{cccc}%
q_{0} & -q_{1} & q_{2} & q_{3}\\
q_{1} & q_{0} & -q_{3} & q_{2}\\
q_{2} & -q_{3} & q_{0} & q_{1}\\
q_{3} & q_{2} & -q_{1} & q_{0}%
\end{array}
\right]
\]
and it is called the right matrix representation of $q.$

\begin{proposition}
\cite{kula} For any $p,q\in\widehat{\mathbb{H}}$ and $r\in%
\mathbb{R}
,$ the followings are satisfied;
\end{proposition}

\begin{description}
\item[i)] $\mathbf{R}_{p+q}=\mathbf{R}_{p}+\mathbf{R}_{q},$

\item[ii)] $\mathbf{R}_{pq}=\mathbf{R}_{q}\mathbf{R}_{p},$

\item[iii)] $\mathbf{R}_{rp}=r\mathbf{R}_{p},$

\item[iv)] $\mathbf{R}_{1}=\mathbf{I}_{4}.$
\end{description}

\section{De Moivre's formula for real matrices of split quaternions}

\noindent Depending on the casual character of split quaternion, we can
express real matrix representations of quaternion as follows:

\noindent\textbf{i.}\qquad Every timelike split quaternion $q.=N_{q}%
(\cosh\theta+\overrightarrow{\varepsilon}\sinh\theta)$ with spacelike vector
part, we may write%
\[
\mathbf{L}_{q}=N_{q}\left[
\begin{array}
[c]{cccc}%
\cosh\theta & -u_{1}\sinh\theta & u_{2}\sinh\theta & u_{3}\sinh\theta\\
u_{1}\sinh\theta & \cosh\theta & u_{3}\sinh\theta & -u_{2}\sinh\theta\\
u_{2}\sinh\theta & u_{3}\sinh\theta & \cosh\theta & -u_{1}\sinh\theta\\
u_{3}\sinh\theta & -u_{2}\sinh\theta & u_{1}\sinh\theta & \cosh\theta
\end{array}
\right]  =N_{q}[\cosh\theta\text{ }\mathbf{I}_{4}+\sinh\theta\text{
}\mathbf{L}_{\overrightarrow{\varepsilon}}],
\]%
\[
\mathbf{R}_{q}=N_{q}\left[
\begin{array}
[c]{cccc}%
\cosh\theta & -u_{1}\sinh\theta & u_{2}\sinh\theta & u_{3}\sinh\theta\\
u_{1}\sinh\theta & \cosh\theta & -u_{3}\sinh\theta & u_{2}\sinh\theta\\
u_{2}\sinh\theta & -u_{3}\sinh\theta & \cosh\theta & u_{1}\sinh\theta\\
u_{3}\sinh\theta & u_{2}\sinh\theta & -u_{1}\sinh\theta & \cosh\theta
\end{array}
\right]  =N_{q}[\cosh\theta\text{ }\mathbf{I}_{4}+\sinh\theta\text{
}\mathbf{R}_{\overrightarrow{\varepsilon}}]
\]
where $\overrightarrow{\varepsilon}=(u_{1},u_{2},u_{3})$ is spacelike unit
vector in $\mathbb{E}_{1}^{3}.$\newpage

\noindent\textbf{ii.}\qquad Every timelike split quaternion $q{\small =}%
N_{q}(\cos\theta+\overrightarrow{\varepsilon}\sin\theta)$ with timelike vector
part, we may write%
\begin{align*}
\mathbf{L}_{q}  &  =N_{q}\left[
\begin{array}
[c]{cccc}%
\cos\theta & -u_{1}\sin\theta & u_{2}\sin\theta & u_{3}\sin\theta\\
u_{1}\sin\theta & \cos\theta & u_{3}\sin\theta & -u_{2}\sin\theta\\
u_{2}\sin\theta & u_{3}\sin\theta & \cos\theta & -u_{1}\sin\theta\\
u_{3}\sin\theta & -u_{2}\sin\theta & u_{1}\sin\theta & \cos\theta
\end{array}
\right]  =N_{q}[\cos\theta\text{ }\mathbf{I}_{4}+\sin\theta\text{ }%
\mathbf{L}_{\overrightarrow{\varepsilon}}],\\
& \\
\mathbf{R}_{q}  &  =N_{q}\left[
\begin{array}
[c]{cccc}%
\cos\theta & -u_{1}\sin\theta & u_{2}\sin\theta & u_{3}\sin\theta\\
u_{1}\sin\theta & \cos\theta & -u_{3}\sin\theta & u_{2}\sin\theta\\
u_{2}\sin\theta & -u_{3}\sin\theta & \cos\theta & u_{1}\sin\theta\\
u_{3}\sin\theta & u_{2}\sin\theta & -u_{1}\sin\theta & \cos\theta
\end{array}
\right]  =N_{q}[\cos\theta\text{ }\mathbf{I}_{4}+\sin\theta\text{ }%
\mathbf{R}_{\overrightarrow{\varepsilon}}]
\end{align*}
where $\overrightarrow{\varepsilon}=(u_{1},u_{2},u_{3})$ is timelike unit
vector in $\mathbb{E}_{1}^{3}.$

\noindent\textbf{iii.}\qquad Every spacelike split quaternion $q=N_{q}%
(\sinh\theta+\overrightarrow{\varepsilon}\cosh\theta),$ we may write
\begin{align*}
\mathbf{L}_{q}  &  =N_{q}\left[
\begin{array}
[c]{cccc}%
\sinh\theta & -u_{1}\cosh\theta & u_{2}\cosh\theta & u_{3}\cosh\theta\\
u_{1}\cosh\theta & \sinh\theta & u_{3}\cosh\theta & -u_{2}\cosh\theta\\
u_{2}\cosh\theta & u_{3}\cosh\theta & \sinh\theta & -u_{1}\cosh\theta\\
u_{3}\cosh\theta & -u_{2}\cosh\theta & u_{1}\cosh\theta & \sinh\theta
\end{array}
\right]  =N_{q}[\sinh\theta\text{ }\mathbf{I}_{4}+\cosh\theta\text{
}\mathbf{L}_{\overrightarrow{\varepsilon}}],\\
& \\
\mathbf{R}_{q}  &  =N_{q}\left[
\begin{array}
[c]{cccc}%
\sinh\theta & -u_{1}\cosh\theta & u_{2}\cosh\theta & u_{3}\cosh\theta\\
u_{1}\cosh\theta & \sinh\theta & -u_{3}\cosh\theta & u_{2}\cosh\theta\\
u_{2}\cosh\theta & -u_{3}\cosh\theta & \sinh\theta & u_{1}\cosh\theta\\
u_{3}\cosh\theta & u_{2}\cosh\theta & -u_{1}\cosh\theta & \sinh\theta
\end{array}
\right]  =N_{q}[\sinh\theta\text{ }\mathbf{I}_{4}+\cosh\theta\text{
}\mathbf{R}_{\overrightarrow{\varepsilon}}]
\end{align*}
where $\overrightarrow{\varepsilon}=(u_{1},u_{2},u_{3})$ is spacelike unit
vector in $\mathbb{E}_{1}^{3}.$

\begin{theorem}
For any timelike split quaternion $q=N_{q}(\cosh\theta+\overrightarrow
{\varepsilon}\sinh\theta)$ with spacelike vector part, we have
\begin{align*}
\left[  \mathbf{L}_{q}\right]  ^{n}  &  =(N_{q})^{n}\left[
\begin{array}
[c]{cccc}%
\cosh n\theta & -u_{1}\sinh n\theta & u_{2}\sinh n\theta & u_{3}\sinh
n\theta\\
u_{1}\sinh n\theta & \cosh n\theta & u_{3}\sinh n\theta & -u_{2}\sinh
n\theta\\
u_{2}\sinh n\theta & u_{3}\sinh n\theta & \cosh n\theta & -u_{1}\sinh
n\theta\\
u_{3}\sinh n\theta & -u_{2}\sinh n\theta & u_{1}\sinh n\theta & \cosh n\theta
\end{array}
\right]  =(N_{q})^{n}[\cosh n\theta\text{ }\mathbf{I}_{4}+\sinh n\theta\text{
}\mathbf{L}_{\overrightarrow{\varepsilon}}],\\
& \\
\left[  \mathbf{R}_{q}\right]  ^{n}  &  =(N_{q})^{n}\left[
\begin{array}
[c]{cccc}%
\cosh n\theta & -u_{1}\sinh n\theta & u_{2}\sinh n\theta & u_{3}\sinh
n\theta\\
u_{1}\sinh n\theta & \cosh n\theta & -u_{3}\sinh n\theta & u_{2}\sinh
n\theta\\
u_{2}\sinh n\theta & -u_{3}\sinh n\theta & \cosh n\theta & u_{1}\sinh
n\theta\\
u_{3}\sinh n\theta & u_{2}\sinh n\theta & -u_{1}\sinh n\theta & \cosh n\theta
\end{array}
\right]  =(N_{q})^{n}[\cosh n\theta\text{ }\mathbf{I}_{4}+\sinh n\theta\text{
}\mathbf{R}_{\overrightarrow{\varepsilon}}]
\end{align*}
where $\overrightarrow{\varepsilon}=(u_{1},u_{2},u_{3})$ is spacelike unit
vector in $\mathbb{E}_{1}^{3}.$
\end{theorem}

\begin{proof}
We use induction on positive integer $n.$ Let $q{\small =}N_{q}(\cosh
\theta+\overrightarrow{\varepsilon}\sinh\theta)$ be any timelike split
quaternion where $\overrightarrow{\varepsilon}=(u_{1},u_{2},u_{3})$ is
spacelike unit vector in $\mathbb{E}_{1}^{3}.$ Assume that
\begin{align*}
\left[  \mathbf{L}_{q}\right]  ^{n}  &  =(N_{q})^{n}\left[
\begin{array}
[c]{cccc}%
\cosh n\theta & {\small -}u_{1}\sinh n\theta & u_{2}\sinh n\theta & u_{3}\sinh
n\theta\\
u_{1}\sinh n\theta & \cosh n\theta & u_{3}\sinh n\theta & {\small -}u_{2}\sinh
n\theta\\
u_{2}\sinh n\theta & u_{3}\sinh n\theta & \cosh n\theta & {\small -}u_{1}\sinh
n\theta\\
u_{3}\sinh n\theta & {\small -}u_{2}\sinh n\theta & u_{1}\sinh n\theta & \cosh
n\theta
\end{array}
\right] \\
& \\
&  =(N_{q})^{n}[\cosh n\theta\text{ }\mathbf{I}_{4}+\sinh n\theta\text{
}\mathbf{L}_{\overrightarrow{\varepsilon}}],
\end{align*}
holds. Then\newline$\newline(\mathbf{L}_{q})^{n+1}{\small =}\left(
\mathbf{L}_{q}\right)  ^{n}$ $\mathbf{L}_{q}$\newline\newline{\small =}%
${\small N}_{q}{}^{n+1}\left[
\begin{array}
[c]{cccc}%
\cosh{\small n\theta} & \text{-}u_{1}\sinh{\small n\theta} & u_{2}%
\sinh{\small n\theta} & u_{3}\sinh{\small n\theta}\\
u_{1}\sinh{\small n\theta} & \cosh{\small n\theta} & u_{3}\sinh{\small n\theta
} & \text{-}u_{2}\sinh{\small n\theta}\\
u_{2}\sinh{\small n\theta} & u_{3}\sinh{\small n\theta} & \cosh{\small n\theta
} & \text{-}u_{1}\sinh{\small n\theta}\\
u_{3}\sinh{\small n\theta} & \text{-}u_{2}\sinh{\small n\theta} & u_{1}%
\sinh{\small n\theta} & \cosh{\small n\theta}%
\end{array}
\right]  \left[
\begin{array}
[c]{cccc}%
\cosh{\small \theta} & \text{-}u_{1}\sinh{\small \theta} & u_{2}%
\sinh{\small \theta} & u_{3}\sinh{\small \theta}\\
u_{1}\sinh{\small \theta} & \cosh{\small \theta} & u_{3}\sinh{\small \theta} &
\text{-}u_{2}\sinh{\small \theta}\\
u_{2}\sinh{\small \theta} & u_{3}\sinh{\small \theta} & \cosh{\small \theta} &
\text{-}u_{1}\sinh{\small \theta}\\
u_{3}\sinh{\small \theta} & \text{-}u_{2}\sinh{\small \theta} & u_{1}%
\sinh{\small \theta} & \cosh{\small \theta}%
\end{array}
\right]  $\newpage\noindent Using the equality%
\[
-u_{1}^{2}+u_{2}^{2}+u_{3}^{2}=1,
\]
and the identities
\begin{align*}
\cosh\theta\cosh n\theta+\sinh\theta\sinh n\theta &  =\cosh(n+1)\theta,\\
& \\
\cosh\theta\sinh n\theta+\sinh\theta\cosh n\theta &  =\sinh(n+1)\theta,
\end{align*}
we get
\begin{align*}
\left(  \mathbf{L}_{q}\right)  ^{n+1}  &  {\small =}(N_{q})^{n+1}\left[
\begin{array}
[c]{cccc}%
\cosh(n{\small +}1)\theta & -u_{1}\sinh(n{\small +}1)\theta & u_{2}%
\sinh(n{\small +}1)\theta & u_{3}\sinh(n{\small +}1)\theta\\
u_{1}\sinh(n{\small +}1)\theta & \cosh(n{\small +}1)\theta & u_{3}%
\sinh(n{\small +}1)\theta & -u_{2}\sinh(n{\small +}1)\theta\\
u_{2}\sinh(n{\small +}1)\theta & u_{3}\sinh(n{\small +}1)\theta &
\cosh(n{\small +}1)\theta & -u_{1}\sinh(n{\small +}1)\theta\\
u_{3}\sinh(n{\small +}1)\theta & -u_{2}\sinh(n{\small +}1)\theta & u_{1}%
\sinh(n{\small +}1)\theta & \cosh(n{\small +}1)\theta
\end{array}
\right] \\
& \\
&  {\small =}(N_{q})^{n+1}\left\{  \cosh(n{\small +}1)\theta\left[
\begin{array}
[c]{cccc}%
1 & 0 & 0 & 0\\
0 & 1 & 0 & 0\\
0 & 0 & 1 & 0\\
0 & 0 & 0 & 1
\end{array}
\right]  +\sinh(n{\small +}1)\theta\left[
\begin{array}
[c]{cccc}%
0 & -u_{1} & u_{2} & u_{3}\\
u_{1} & 0 & u_{3} & -u_{2}\\
u_{2} & u_{3} & 0 & -u_{1}\\
u_{3} & -u_{2} & u_{1} & 0
\end{array}
\right]  \right\} \\
& \\
&  =(N_{q})^{n+1}\left[  \cosh(n{\small +}1)\theta\text{ }\mathbf{I}_{4}%
+\sinh(n{\small +}1)\theta\text{ }\mathbf{L}_{\overrightarrow{\varepsilon}%
}\right]  .
\end{align*}
Hence the first formula is true. Similarly suppose that
\begin{align*}
\left(  \mathbf{R}_{q}\right)  ^{n}  &  =(N_{q})^{n}\left[
\begin{array}
[c]{cccc}%
\cosh n\theta & -u_{1}\sinh n\theta & u_{2}\sinh n\theta & u_{3}\sinh
n\theta\\
u_{1}\sinh n\theta & \cosh n\theta & -u_{3}\sinh n\theta & u_{2}\sinh
n\theta\\
u_{2}\sinh n\theta & -u_{3}\sinh n\theta & \cosh n\theta & u_{1}\sinh
n\theta\\
u_{3}\sinh n\theta & u_{2}\sinh n\theta & -u_{1}\sinh n\theta & \cosh n\theta
\end{array}
\right] \\
& \\
&  =(N_{q})^{n}[\cosh n\theta\text{ }\mathbf{I}_{4}+\sinh n\theta\text{
}\mathbf{R}_{\overrightarrow{\varepsilon}}]
\end{align*}
is true. Then\newline$\newline\left(  \mathbf{R}_{q}\right)  ^{n+1}=\left(
\mathbf{R}_{q}\right)  ^{n}$ $\mathbf{R}_{q}$\newline\newline{\small =}%
$N_{q}^{n+1}\left[
\begin{array}
[c]{cccc}%
\cosh n\theta & \text{-}u_{1}\sinh n\theta & u_{2}\sinh n\theta & u_{3}\sinh
n\theta\\
u_{1}\sinh n\theta & \cosh n\theta & \text{-}u_{3}\sinh n\theta & u_{2}\sinh
n\theta\\
u_{2}\sinh n\theta & \text{-}u_{3}\sinh n\theta & \cosh n\theta & u_{1}\sinh
n\theta\\
u_{3}\sinh n\theta & u_{2}\sinh n\theta & \text{-}u_{1}\sinh n\theta & \cosh
n\theta
\end{array}
\right]  \left[
\begin{array}
[c]{cccc}%
\cosh\theta & \text{-}u_{1}\sinh\theta & u_{2}\sinh\theta & u_{3}\sinh\theta\\
u_{1}\sinh\theta & \cosh\theta & \text{-}u_{3}\sinh\theta & u_{2}\sinh\theta\\
u_{2}\sinh\theta & \text{-}u_{3}\sinh\theta & \cosh\theta & u_{1}\sinh\theta\\
u_{3}\sinh\theta & u_{2}\sinh\theta & \text{-}u_{1}\sinh\theta & \cosh\theta
\end{array}
\right]  \newline\newline$Then we get
\begin{align*}
\left(  \mathbf{R}_{q}\right)  ^{n+1}  &  \text{=}(N_{q})^{n+1}\left[
\begin{array}
[c]{cccc}%
\cosh(n\text{+}1)\theta & \text{-}u_{1}\sinh(n\text{+}1)\theta & u_{2}%
\sinh(n\text{+}1)\theta & u_{3}\sinh(n\text{+}1)\theta\\
u_{1}\sinh(n\text{+}1)\theta & \cosh(n\text{+}1)\theta & \text{-}u_{3}%
\sinh(n\text{+}1)\theta & u_{2}\sinh(n\text{+}1)\theta\\
u_{2}\sinh(n\text{+}1)\theta & \text{-}u_{3}\sinh(n\text{+}1)\theta &
\cosh(n\text{+}1)\theta & u_{1}\sinh(n\text{+}1)\theta\\
u_{3}\sinh(n\text{+}1)\theta & u_{2}\sinh(n\text{+}1)\theta & \text{-}%
u_{1}\sinh(n\text{+}1)\theta & \cosh(n\text{+}1)\theta
\end{array}
\right] \\
& \\
&  \text{=}(N_{q})^{n+1}\left\{  \cosh(n\text{+}1)\theta\left[
\begin{array}
[c]{cccc}%
1 & 0 & 0 & 0\\
0 & 1 & 0 & 0\\
0 & 0 & 1 & 0\\
0 & 0 & 0 & 1
\end{array}
\right]  \text{+}\sinh(n\text{+}1)\theta\left[
\begin{array}
[c]{cccc}%
0 & \text{-}u_{1} & u_{2} & u_{3}\\
u_{1} & 0 & \text{-}u_{3} & u_{2}\\
u_{2} & \text{-}u_{3} & 0 & u_{1}\\
u_{3} & u_{2} & \text{-}u_{1} & 0
\end{array}
\right]  \right\} \\
& \\
&  \text{=}(N_{q})^{n+1}\left[  \cosh(n+1)\theta\text{ }\mathbf{I}_{4}%
\text{+}\sinh(n+1)\theta\text{ }\mathbf{R}_{\overrightarrow{\varepsilon}%
}\right]  .
\end{align*}
So, the second formula is also true.
\end{proof}

\newpage

\begin{theorem}
For any timelike split quaternion $q$=$N_{q}(\cos\theta$+$\overrightarrow
{\varepsilon}\sin\theta)$ with timelike vector part,
\begin{align*}
\left(  \mathbf{L}_{q}\right)  ^{n}  &  \text{=}(N_{q})^{n}\left[
\begin{array}
[c]{cccc}%
\cos n\theta & \text{-}u_{1}\sin n\theta & u_{2}\sin n\theta & u_{3}\sin
n\theta\\
u_{1}\sin n\theta & \cos n\theta & u_{3}\sin n\theta & \text{-}u_{2}\sin
n\theta\\
u_{2}\sin n\theta & u_{3}\sin n\theta & \cos n\theta & \text{-}u_{1}\sin
n\theta\\
u_{3}\sin n\theta & \text{-}u_{2}\sin n\theta & u_{1}\sin n\theta & \cos
n\theta
\end{array}
\right]  \text{=}(N_{q})^{n}[\cos n\theta\text{ }\mathbf{I}_{4}\text{+}\sin
n\theta\text{ }\mathbf{L}_{\overrightarrow{\varepsilon}}],\\
& \\
\left(  \mathbf{R}_{q}\right)  ^{n}  &  \text{=}(N_{q})^{n}\left[
\begin{array}
[c]{cccc}%
\cos n\theta & \text{-}u_{1}\sin n\theta & u_{2}\sin n\theta & u_{3}\sin
n\theta\\
u_{1}\sin n\theta & \cos n\theta & \text{-}u_{3}\sin n\theta & u_{2}\sin
n\theta\\
u_{2}\sin n\theta & \text{-}u_{3}\sin n\theta & \cos n\theta & u_{1}\sin
n\theta\\
u_{3}\sin n\theta & u_{2}\sin n\theta & \text{-}u_{1}\sin n\theta & \cos
n\theta
\end{array}
\right]  \text{=}(N_{q})^{n}[\cos n\theta\text{ }\mathbf{I}_{4}\text{+}\sin
n\theta\text{ }\mathbf{R}_{\overrightarrow{\varepsilon}}]
\end{align*}
where $\overrightarrow{\varepsilon}=(u_{1},u_{2},u_{3})$ is timelike unit
vector in $\mathbb{E}_{1}^{3}.$
\end{theorem}

\begin{proof}
We use induction on positive integer $n.$ For any timelike split quaternion
$q$=$N_{q}(\cos\theta$+$\overrightarrow{\varepsilon}\sin\theta)$ with timelike
unit vector $\overrightarrow{\varepsilon}=(u_{1},u_{2},u_{3})$ in
$\mathbb{E}_{1}^{3}.$ Suppose that
\begin{align*}
\left(  \mathbf{L}_{q}\right)  ^{n}  &  =(N_{q})^{n}\left[
\begin{array}
[c]{cccc}%
\cos n\theta & \text{-}u_{1}\sin n\theta & u_{2}\sin n\theta & u_{3}\sin
n\theta\\
u_{1}\sin n\theta & \cos n\theta & u_{3}\sin n\theta & \text{-}u_{2}\sin
n\theta\\
u_{2}\sin n\theta & u_{3}\sin n\theta & \cos n\theta & \text{-}u_{1}\sin
n\theta\\
u_{3}\sin n\theta & \text{-}u_{2}\sin n\theta & u_{1}\sin n\theta & \cos
n\theta
\end{array}
\right] \\
& \\
&  =(N_{q})^{n}[\cos n\theta\text{ }\mathbf{I}_{4}\text{+}\sin n\theta\text{
}\mathbf{L}_{\overrightarrow{\varepsilon}}],
\end{align*}
holds. Then\newline$\newline(\mathbf{L}_{q})^{n+1}=\left(  \mathbf{L}%
_{q}\right)  ^{n}$ $\mathbf{L}_{q}$\newline$\newline$=${\small (N}%
_{q}{\small )}^{n+1}\left[
\begin{array}
[c]{cccc}%
\cos n\theta & \text{-}u_{1}\sin n\theta & u_{2}\sin n\theta & u_{3}\sin
n\theta\\
u_{1}\sin n\theta & \cos n\theta & u_{3}\sin n\theta & \text{-}u_{2}\sin
n\theta\\
u_{2}\sin n\theta & u_{3}\sin n\theta & \cos n\theta & \text{-}u_{1}\sin
n\theta\\
u_{3}\sin n\theta & \text{-}u_{2}\sin n\theta & u_{1}\sin n\theta & \cos
n\theta
\end{array}
\right]  \left[
\begin{array}
[c]{cccc}%
\cos\theta & \text{-}u_{1}\sin\theta & u_{2}\sin\theta & u_{3}\sin\theta\\
u_{1}\sin\theta & \cos\theta & u_{3}\sin\theta & \text{-}u_{2}\sin\theta\\
u_{2}\sin\theta & u_{3}\sin\theta & \cos\theta & \text{-}u_{1}\sin\theta\\
u_{3}\sin\theta & \text{-}u_{2}\sin\theta & u_{1}\sin\theta & \cos\theta
\end{array}
\right]  \newline\newline$Using the equality%
\[
-u_{1}^{2}+u_{2}^{2}+u_{3}^{2}=-1,
\]
and the identities%
\begin{align*}
\cos\theta\cos n\theta-\sin\theta\sin n\theta &  =\cos(n+1)\theta,\\
& \\
\cos\theta\sin n\theta+\sin\theta\cos n\theta &  =\sin(n+1)\theta,
\end{align*}
we get
\begin{align*}
\left(  \mathbf{L}_{q}\right)  ^{n+1}  &  =(N_{q})^{n+1}\left[
\begin{array}
[c]{cccc}%
\cos(n+1)\theta & \text{-}u_{1}\sin(n+1)\theta & u_{2}\sin(n+1)\theta &
u_{3}\sin(n+1)\theta\\
u_{1}\sin(n+1)\theta & \cos(n+1)\theta & u_{3}\sin(n+1)\theta & \text{-}%
u_{2}\sin(n+1)\theta\\
u_{2}\sin(n+1)\theta & u_{3}\sin(n+1)\theta & \cos(n+1)\theta & \text{-}%
u_{1}\sin(n+1)\theta\\
u_{3}\sin(n+1)\theta & \text{-}u_{2}\sin(n+1)\theta & u_{1}\sin(n+1)\theta &
\cos(n+1)\theta
\end{array}
\right] \\
& \\
&  =(N_{q})^{n+1}\left\{  \cos(n+1)\theta\left[
\begin{array}
[c]{cccc}%
1 & 0 & 0 & 0\\
0 & 1 & 0 & 0\\
0 & 0 & 1 & 0\\
0 & 0 & 0 & 1
\end{array}
\right]  +\sin(n+1)\theta\left[
\begin{array}
[c]{cccc}%
0 & \text{-}u_{1} & u_{2} & u_{3}\\
u_{1} & 0 & u_{3} & \text{-}u_{2}\\
u_{2} & u_{3} & 0 & \text{-}u_{1}\\
u_{3} & \text{-}u_{2} & u_{1} & 0
\end{array}
\right]  \right\} \\
& \\
&  =(N_{q})^{n+1}\left[  \cos(n+1)\theta\text{ }\mathbf{I}_{4}+\sin
(n+1)\theta\text{ }\mathbf{L}_{\overrightarrow{\varepsilon}}\right]  .
\end{align*}
Hence the first formula is true.$\newline\newline$Assume that
\begin{align*}
\left(  \mathbf{R}_{q}\right)  ^{n}  &  =(N_{q})^{n}\left[
\begin{array}
[c]{cccc}%
\cos n\theta & \text{-}u_{1}\sin n\theta & u_{2}\sin n\theta & u_{3}\sin
n\theta\\
u_{1}\sin n\theta & \cos n\theta & \text{-}u_{3}\sin n\theta & u_{2}\sin
n\theta\\
u_{2}\sin n\theta & \text{-}u_{3}\sin n\theta & \cos n\theta & u_{1}\sin
n\theta\\
u_{3}\sin n\theta & u_{2}\sin n\theta & \text{-}u_{1}\sin n\theta & \cos
n\theta
\end{array}
\right] \\
& \\
&  =(N_{q})^{n}[\cos n\theta\text{ }\mathbf{I}_{4}\text{+}\sin n\theta\text{
}\mathbf{R}_{\overrightarrow{\varepsilon}}]
\end{align*}
is true. Then\newline$\newline\left(  \mathbf{R}_{q}\right)  ^{n+1}=\left(
\mathbf{R}_{q}\right)  ^{n}$ $\mathbf{R}_{q}$\newline\newline=$(N_{q}%
)^{n+1}\left[
\begin{array}
[c]{cccc}%
\cos n\theta & \text{{\small -}}u_{1}\sin n\theta & u_{2}\sin n\theta &
u_{3}\sin n\theta\\
u_{1}\sin n\theta & \cos n\theta & \text{{\small -}}u_{3}\sin n\theta &
u_{2}\sin n\theta\\
u_{2}\sin n\theta & \text{{\small -}}u_{3}\sin n\theta & \cos n\theta &
u_{1}\sin n\theta\\
u_{3}\sin n\theta & u_{2}\sin n\theta & \text{{\small -}}u_{1}\sin n\theta &
\cos n\theta
\end{array}
\right]  \left[
\begin{array}
[c]{cccc}%
\cos\theta & \text{{\small -}}u_{1}\sin\theta & u_{2}\sin\theta & u_{3}%
\sin\theta\\
u_{1}\sin\theta & \cos\theta & \text{{\small -}}u_{3}\sin\theta & u_{2}%
\sin\theta\\
u_{2}\sin\theta & \text{{\small -}}u_{3}\sin\theta & \cos\theta & u_{1}%
\sin\theta\\
u_{3}\sin\theta & u_{2}\sin\theta & \text{{\small -}}u_{1}\sin\theta &
\cos\theta
\end{array}
\right]  $\newline$\newline$Then we get
\begin{align*}
\left(  \mathbf{R}_{q}\right)  ^{n+1}  &  \text{=}(N_{q})^{n+1}\left[
\begin{array}
[c]{cccc}%
\cos(n{\small +}1)\theta & \text{{\small -}}u_{1}\sin(n{\small +}1)\theta &
u_{2}\sin(n{\small +}1)\theta & u_{3}\sin(n{\small +}1)\theta\\
u_{1}\sin(n{\small +}1)\theta & \cos(n{\small +}1)\theta & \text{{\small -}%
}u_{3}\sin(n{\small +}1)\theta & u_{2}\sin(n{\small +}1)\theta\\
u_{2}\sin(n{\small +}1)\theta & \text{{\small -}}u_{3}\sin(n{\small +}1)\theta
& \cos(n{\small +}1)\theta & u_{1}\sin(n{\small +}1)\theta\\
u_{3}\sin(n{\small +}1)\theta & u_{2}\sin(n{\small +}1)\theta &
\text{{\small -}}u_{1}\sin(n{\small +}1)\theta & \cos(n{\small +}1)\theta
\end{array}
\right] \\
& \\
&  \text{=}(N_{q})^{n{\small +}1}\left\{  \cos(n{\small +}1)\theta\left[
\begin{array}
[c]{cccc}%
1 & 0 & 0 & 0\\
0 & 1 & 0 & 0\\
0 & 0 & 1 & 0\\
0 & 0 & 0 & 1
\end{array}
\right]  {\small +}\sin(n{\small +}1)\theta\left[
\begin{array}
[c]{cccc}%
0 & \text{{\small -}}u_{1} & u_{2} & u_{3}\\
u_{1} & 0 & \text{{\small -}}u_{3} & u_{2}\\
u_{2} & \text{-}u_{3} & 0 & u_{1}\\
u_{3} & u_{2} & \text{-}u_{1} & 0
\end{array}
\right]  \right\} \\
& \\
&  \text{=}(N_{q})^{n{\small +}1}\left[  \cos(n{\small +}1)\theta\text{
}\mathbf{I}_{4}\text{+}\sin(n{\small +}1)\theta\text{ }\mathbf{R}%
_{\overrightarrow{\varepsilon}}\right]  .
\end{align*}
So, the second formula is also true.
\end{proof}

\begin{theorem}
For any spacelike split quaternion $q=N_{q}(\sinh\theta${\small +}%
$\overrightarrow{\varepsilon}\cosh\theta)$,\newline If $n$ odd then%
\begin{align*}
\left(  \mathbf{L}_{q}\right)  ^{n}  &  =(N_{q})^{n}\left[
\begin{array}
[c]{cccc}%
\sinh{\small n\theta} & \text{-}u_{1}\cosh{\small n\theta} & u_{2}%
\cosh{\small n\theta} & u_{3}\cosh{\small n\theta}\\
u_{1}\cosh{\small n\theta} & \sinh{\small n\theta} & u_{3}\cosh{\small n\theta
} & \text{-}u_{2}\cosh{\small n\theta}\\
u_{2}\cosh{\small n\theta} & u_{3}\cosh{\small n\theta} & \sinh{\small n\theta
} & \text{-}u_{1}\cosh{\small n\theta}\\
u_{3}\cosh{\small n\theta} & \text{-}u_{2}\cosh{\small n\theta} & u_{1}%
\cosh{\small n\theta} & \sinh{\small n\theta}%
\end{array}
\right]  =(N_{q})^{n}[\sinh{\small n\theta}\text{ }\mathbf{I}_{4}%
{\small +}\cosh{\small n\theta}\text{ }\mathbf{L}_{\overrightarrow
{\varepsilon}}],\\
& \\
\left(  \mathbf{R}_{q}\right)  ^{n}  &  =(N_{q})^{n}\left[
\begin{array}
[c]{cccc}%
\sinh{\small n\theta} & \text{-}u_{1}\cosh{\small n\theta} & u_{2}%
\cosh{\small n\theta} & u_{3}\cosh{\small n\theta}\\
u_{1}\cosh{\small n\theta} & \sinh{\small n\theta} & \text{-}u_{3}%
\cosh{\small n\theta} & u_{2}\cosh{\small n\theta}\\
u_{2}\cosh{\small n\theta} & \text{-}u_{3}\cosh{\small n\theta} &
\sinh{\small n\theta} & u_{1}\cosh{\small n\theta}\\
u_{3}\cosh{\small n\theta} & u_{2}\cosh{\small n\theta} & \text{-}u_{1}%
\cosh{\small n\theta} & \sinh{\small n\theta}%
\end{array}
\right]  =(N_{q})^{n}[\sinh{\small n\theta}\text{ }\mathbf{I}_{4}%
+\cosh{\small n\theta}\text{ }\mathbf{L}_{\overrightarrow{\varepsilon}}].
\end{align*}
If $n$ is even then
\begin{align*}
\left(  \mathbf{L}_{q}\right)  ^{n}\text{{}}  &  \text{{\small =}}(N_{q}%
)^{n}\left[
\begin{array}
[c]{cccc}%
\cosh{\small n\theta} & \text{-}u_{1}\sinh{\small n\theta} & u_{2}%
\sinh{\small n\theta} & u_{3}\sinh{\small n\theta}\\
u_{1}\sinh{\small n\theta} & \cosh{\small n\theta} & u_{3}\sinh{\small n\theta
} & \text{-}u_{2}\sinh{\small n\theta}\\
u_{2}\sinh{\small n\theta} & u_{3}\sinh{\small n\theta} & \cosh{\small n\theta
} & \text{-}u_{1}\sinh{\small n\theta}\\
u_{3}\sinh{\small n\theta} & \text{-}u_{2}\sinh{\small n\theta} & u_{1}%
\sinh{\small n\theta} & \cosh{\small n\theta}%
\end{array}
\right]  {\small =}(N_{q})^{n}[\cosh{\small n\theta}\text{ }\mathbf{I}%
_{4}{\small +}\sinh{\small n\theta}\text{ }\mathbf{L}_{\overrightarrow
{\varepsilon}}],\\
& \\
\left(  \mathbf{R}_{q}\right)  ^{n}  &  \text{=}(N_{q})^{n}\left[
\begin{array}
[c]{cccc}%
\cosh{\small n\theta} & \text{-}u_{1}\sinh{\small n\theta} & u_{2}%
\sinh{\small n\theta} & u_{3}\sinh{\small n\theta}\\
u_{1}\sinh{\small n\theta} & \cosh{\small n\theta} & \text{-}u_{3}%
\sinh{\small n\theta} & u_{2}\sinh{\small n\theta}\\
u_{2}\sinh{\small n\theta} & \text{-}u_{3}\sinh{\small n\theta} &
\cosh{\small n\theta} & u_{1}\sinh{\small n\theta}\\
u_{3}\sinh{\small n\theta} & u_{2}\sinh{\small n\theta} & \text{-}u_{1}%
\sinh{\small n\theta} & \cosh{\small n\theta}%
\end{array}
\right]  =(N_{q})^{n}[\cosh{\small n\theta}\text{ }\mathbf{I}_{4}%
{\small +}\sinh{\small n\theta}\text{ }\mathbf{L}_{\overrightarrow
{\varepsilon}}].
\end{align*}
where $\overrightarrow{\varepsilon}=(u_{1},u_{2},u_{3})$ is spacelike unit
vector in $\mathbb{E}_{1}^{3}.$\newpage
\end{theorem}

\section{Euler formula for real matrices of split quaternions}

\noindent In this part, we will state Euler formula for real matrices of
spacelike and timelike pure split quaternions, separately.

\begin{theorem}
For any timelike pure unit split quaternion $q=\overrightarrow{\varepsilon
}=(u_{1},u_{2},u_{3}),$
\[
e^{\theta\mathbf{L}_{q}}=\cos\theta\mathbf{I}_{4}+\sin\theta\text{ }%
\mathbf{L}_{q}\text{ and }e^{\theta\mathbf{R}_{q}}=\cos\theta\mathbf{I}%
_{4}+\sin\theta\text{ }\mathbf{R}_{q}.
\]

\end{theorem}

\begin{proof}
Let $q=\overrightarrow{\varepsilon}=(u_{1},u_{2},u_{3})$ be any timelike pure
unit split quaternion, we have
\[
\mathbf{L}_{q}=\left[
\begin{array}
[c]{cccc}%
0 & -u_{1} & u_{2} & u_{3}\\
u_{1} & 0 & u_{3} & -u_{2}\\
u_{2} & u_{3} & 0 & -u_{1}\\
u_{3} & -u_{2} & u_{1} & 0
\end{array}
\right]  \text{ and }\left(  \mathbf{L}_{q}\right)  ^{2}=(-u_{1}^{2}+u_{2}%
^{2}+u_{3}^{2})\left[
\begin{array}
[c]{cccc}%
1 & 0 & 0 & 0\\
0 & 1 & 0 & 0\\
0 & 0 & 1 & 0\\
0 & 0 & 0 & 1
\end{array}
\right]
\]
Using the identity%
\[
-u_{1}^{2}+u_{2}^{2}+u_{3}^{2}=-1
\]
we get $\left(  \mathbf{L}_{q}\right)  ^{2}=-\mathbf{I}_{4}.$ So, we have
\begin{align*}
e^{\theta\mathbf{L}_{q}}  &  =\mathbf{I}_{4}+\theta\mathbf{L}_{q}%
+\dfrac{\theta^{2}}{2!}(\mathbf{L}_{q})^{2}+\dfrac{\theta^{3}}{3!}%
(\mathbf{L}_{q})^{3}+\dfrac{\theta^{4}}{4!}(\mathbf{L}_{q})^{4}+\dfrac
{\theta^{5}}{5!}(\mathbf{L}_{q})^{5}+\dfrac{\theta^{6}}{6!}(\mathbf{L}%
_{q})^{6}+\cdots\\
& \\
&  =\mathbf{I}_{4}+\theta\mathbf{L}_{q}+\dfrac{\theta^{2}}{2!}(-\mathbf{I}%
_{4})+\dfrac{\theta^{3}}{3!}(-\mathbf{L}_{q})+\dfrac{\theta^{4}}%
{4!}(\mathbf{I}_{4})+\dfrac{\theta^{5}}{5!}(\mathbf{L}_{q})+\dfrac{\theta^{6}%
}{6!}(-\mathbf{I}_{4})+\cdots\\
& \\
&  =(1-\dfrac{\theta^{2}}{2!}+\dfrac{\theta^{4}}{4!}-\dfrac{\theta^{6}}%
{6!}+\cdots)\mathbf{I}_{4}+(\theta-\dfrac{\theta^{3}}{3!}+\dfrac{\theta^{5}%
}{5!}-\cdots)\mathbf{L}_{q}\\
& \\
&  =\cos\theta\text{ }\mathbf{I}_{4}+\sin\theta\text{ }\mathbf{L}_{q}.
\end{align*}
It is easy to see $\left(  \mathbf{R}_{q}\right)  ^{2}=-\mathbf{I}_{4}$ and
$e^{\theta\mathbf{R}_{q}}=\cos\theta\mathbf{I}_{4}+\sin\theta$ $\mathbf{R}%
_{q}.$
\end{proof}

\begin{theorem}
For any spacelike \ unit pure split quaternion $q=\overrightarrow{\varepsilon
}=(u_{1},u_{2},u_{3}),$%
\[
e^{\theta\mathbf{L}_{q}}=\cosh\theta\text{ }\mathbf{I}_{4}+\sinh\theta\text{
}\mathbf{L}_{q}\text{ and }e^{\theta\mathbf{R}_{q}}=\cosh\theta\text{
}\mathbf{I}_{4}+\sinh\theta\text{ }\mathbf{R}_{q}%
\]

\end{theorem}

\begin{proof}
Let $q=\overrightarrow{\varepsilon}=(u_{1},u_{2},u_{3})$ be any spacelike pure
unit split quaternion, we have
\[
\left(  \mathbf{L}_{q}\right)  ^{2}=(-u_{1}^{2}+u_{2}^{2}+u_{3}^{2})\left[
\begin{array}
[c]{cccc}%
1 & 0 & 0 & 0\\
0 & 1 & 0 & 0\\
0 & 0 & 1 & 0\\
0 & 0 & 0 & 1
\end{array}
\right]
\]
Using the identity%
\[
-u_{1}^{2}+u_{2}^{2}+u_{3}^{2}=1
\]
we get $\left(  \mathbf{L}_{q}\right)  ^{2}=\mathbf{I}_{4}.$ So, we have
\begin{align*}
e^{\theta\mathbf{L}_{q}}  &  =\mathbf{I}_{4}+\theta\mathbf{L}_{q}%
+\dfrac{\theta^{2}}{2!}(\mathbf{L}_{q})^{2}+\dfrac{\theta^{3}}{3!}%
(\mathbf{L}_{q})^{3}+\dfrac{\theta^{4}}{4!}(\mathbf{L}_{q})^{4}+\dfrac
{\theta^{5}}{5!}(\mathbf{L}_{q})^{5}+\dfrac{\theta^{6}}{6!}(\mathbf{L}%
_{q})^{6}+\cdots\\
& \\
&  =\mathbf{I}_{4}+\theta\mathbf{L}_{q}+\dfrac{\theta^{2}}{2!}\mathbf{I}%
_{4}+\dfrac{\theta^{3}}{3!}\mathbf{L}_{q}+\dfrac{\theta^{4}}{4!}\mathbf{I}%
_{4}+\dfrac{\theta^{5}}{5!}\mathbf{L}_{q}+\dfrac{\theta^{6}}{6!}\mathbf{I}%
_{4}+\cdots\\
& \\
&  =(1+\dfrac{\theta^{2}}{2!}+\dfrac{\theta^{4}}{4!}+\dfrac{\theta^{6}}%
{6!}+\cdots)\mathbf{I}_{4}+(\theta+\dfrac{\theta^{3}}{3!}+\dfrac{\theta^{5}%
}{5!}+\cdots)\mathbf{L}_{q}\\
& \\
&  =\cosh\theta\text{ }\mathbf{I}_{4}+\sinh\theta\text{ }\mathbf{L}_{q}.
\end{align*}
Similarly, it can be easily shown that $\left(  \mathbf{R}_{q}\right)
^{2}=\mathbf{I}_{4}$ and $e^{\theta\mathbf{R}_{q}}=\cosh\theta\mathbf{I}%
_{4}+\sinh\theta$ $\mathbf{R}_{q}.$
\end{proof}

\bigskip

Melek Erdo\u{g}du

Department of Mathematics-Computer Sciences

Necmettin Erbakan University

42060 Konya, Turkey

e-mail: merdogdu@konya.edu.tr.

\bigskip

Mustafa \"{O}zdemir

Department of Mathematics

Akdeniz University

07058 Antalya, Turkey

e-mail: mozdemir@akdeniz.edu.tr.

\end{document}